\documentclass[12pt,reqno]{amsart}

\usepackage{amsmath,amsthm,amssymb,comment,fullpage}
\usepackage{braket}
\usepackage{mathtools}

\usepackage{times}
\usepackage[T1]{fontenc}
\usepackage{mathrsfs}
\usepackage{latexsym}
\usepackage[dvips]{graphics}
\usepackage{epsfig}
\usepackage{amsmath,amsfonts,amsthm,amssymb,amscd}
\input amssym.def
\input amssym.tex
\usepackage{color}
\usepackage{hyperref}
\usepackage{url}
\newcommand{\bburl}[1]{\textcolor{blue}{\url{#1}}}

\usepackage{url}
\usepackage{tikz}
\usepackage{tkz-tab}
\usepackage{tkz-graph}
\usetikzlibrary{shapes.geometric,positioning}

\newcommand{\burl}[1]{\textcolor{blue}{\url{#1}}}

\numberwithin{equation}{section}

\newtheorem{thm}{Theorem}[section]

\newtheorem{lem}[thm]{Lemma}
\newtheorem{prop}[thm]{Proposition}

\newtheorem{defi}[thm]{Definition}

\theoremstyle{plain}

\newtheorem{definition}[thm]{Definition}

\newtheorem{lemma}[thm]{Lemma}

\newtheorem{theorem}[thm]{Theorem}

\newcommand\be{\begin{equation}}
\newcommand\ee{\end{equation}}
\newcommand\bee{\begin{equation*}}
\newcommand\eee{\end{equation*}}
\newcommand\bea{\begin{eqnarray}}
\newcommand\eea{\end{eqnarray}}
\newcommand\beae{\begin{eqnarray*}}
\newcommand\eeae{\end{eqnarray*}}
\newcommand\bi{\begin{itemize}}
\newcommand\ei{\end{itemize}}
\newcommand\ben{\begin{enumerate}}
\newcommand\een{\end{enumerate}}
\newcommand\bc{\begin{center}}
\newcommand\ec{\end{center}}
\newcommand\ba{\begin{array}}
\newcommand\ea{\end{array}}

\newcommand\frakfamily{\usefont{U}{yfrak}{m}{n}}
\DeclareTextFontCommand{\textfrak}{\frakfamily}

\newtheorem{rek}[thm]{Remark}

\newcommand{\hr}[1]{\href{#1}{\url{#1}}}

\newcommand{\dfq}{d_{\rm FQ}}
\newcommand{\dave}{d_{\rm FQ; ave}}

\title{Legal Decompositions Arising from Non-positive Linear Recurrences}

\author{Minerva Catral}
\email{\textcolor{blue}{\href{mailto:catralm@xavier.edu}{catralm@xavier.edu}}}
\address{Department of Mathematics, Xavier University, Cincinnati, OH 45207}

\author{Pari L. Ford}
\email{\textcolor{blue}{\href{mailto:fordpl@bethanylb.edu}{fordpl@bethanylb.edu}}}
\address{Department of Mathematics and Physics, Bethany College, Lindsborg, KS 67456 }

\author{Pamela E. Harris}
\email{\textcolor{blue}{\href{mailto:Pamela.Harris@usma.edu}{Pamela.Harris@usma.edu}}}
\address{Department of Mathematical Sciences, United States Military Academy, West Point, NY 10996}

\author{Steven J. Miller}
\email{\textcolor{blue}{\href{mailto:sjm1@williams.edu}{sjm1@williams.edu}},  \textcolor{blue}{\href{Steven.Miller.MC.96@aya.yale.edu}{Steven.Miller.MC.96@aya.yale.edu}}}
\address{Department of Mathematics and Statistics, Williams College, Williamstown, MA 01267}

\author{Dawn Nelson}
\email{\textcolor{blue}{\href{mailto:dnelson1@saintpeters.edu}{dnelson1@saintpeters.edu}}}
\address{Department of Mathematics, Saint Peter's University, Jersey City, NJ 07306}

\thanks{The fourth named author was partially supported by NSF grants DMS1265673 and DMS1561945. This research was performed while
the third named author held a National Research Council Research Associateship Award at USMA/ARL. This work was begun at the 2014 REUF Meeting at AIM and continued during a REUF continuation grant at ICERM; it is a pleasure to thank them for  their support. We also thank the participants  at the 13\textsuperscript{th} annual Combinatorial and Additive Number Theory workshop (CANT 2015) for helpful discussions.}



\date{\today}

\begin{document}

\maketitle

\begin{abstract} Zeckendorf's theorem states that any positive integer can be written uniquely as a sum of non-adjacent Fibonacci numbers; this result has been generalized to many recurrence relations, especially those arising from linear recurrences with leading term positive. We investigate legal decompositions arising from two new sequences: the $(s,b)$-Generacci sequence and the Fibonacci Quilt sequence. Both satisfy recurrence relations with leading term zero, and thus previous results and techniques do not apply.  These sequences exhibit drastically different behavior.  We show that  the $(s,b)$-Generacci sequence leads to unique legal decompositions,  whereas not only do we have non-unique legal decompositions with  the Fibonacci Quilt sequence, we also have that  in this case the average number of legal decompositions  grows exponentially. Another interesting difference is that  while in the $(s,b)$-Generacci case the greedy algorithm always leads to a legal decomposition,   in the Fibonacci Quilt setting  the greedy algorithm leads to a legal decomposition (approximately) 93\% of the time. In the $(s,b)$-Generacci case, we again have Gaussian behavior in the number of summands as well as for the Fibonacci Quilt sequence when we restrict to decompositions resulting from a modified greedy algorithm.
\end{abstract}

\tableofcontents


\section{Introduction}


A beautiful result of Zeckendorf describes the Fibonacci numbers as the unique sequence from which every natural number can be expressed uniquely as a sum of nonconsecutive terms in the sequence \cite{Ze}.  Zeckendorf's theorem inspired many questions about this decomposition, and generalizations of the notions of legal decompositions of natural numbers as sums of elements from an integer sequence has been a fruitful area of research \cite{BBGILMT, BCCSW, BDEMMTTW, BILMT, CFHMN1, DDKMMV, DDKMV, DFFHMPP, GTNP, Ha, KKMW, MW1, MW2}.

Much of previous work has focused on sequences given by a \textit{Positive Linear Recurrence (PLR)}, which are sequences where there is a fixed depth $L > 0$ and non-negative integers $c_1, \dots, c_L$ with $c_1, c_L $ non-zero such that \be a_{n+1} \ = \ c_1 a_n + \cdots + c_L a_{n+1-L}. \ee   The restriction that $c_1 > 0$ is required to gain needed control over roots of polynomials associated to the characteristic polynomials of the recurrence and related generating functions, though in the companion paper \cite{CFHMNPX} we show how to bypass some of these technicalities through new combinatorial techniques. The motivation for this paper is to investigate whether the positivity of the first coefficient is needed solely to simplify the arguments, or if fundamentally different behavior can emerge if the said condition is not met.  To this end, we investigate the legal decompositions arising from two different sequences which we introduce in this paper:  the $(s,b)$-Generacci sequence and the Fibonacci Quilt sequence. Both satisfy recurrence relations with leading term zero,  hence previous  results  and techniques are not applicable. Moreover, although both have non-positive linear recurrences (as their leading term is zero), they exhibit drastically different behavior: the $(s,b)$-Generacci sequence leads to unique legal decompositions,  whereas not only do we have non-unique legal decompositions with  the Fibonacci Quilt sequence, we also have that the average number of legal decompositions grows exponentially. Another interesting difference is that  while in the $(s,b)$-Generacci case the greedy algorithm always leads to a legal decomposition,   in the Fibonacci Quilt setting  the greedy algorithm leads to a legal decomposition (approximately) 93\% of the time.

We conclude the introduction by first describing the two sequences and their resulting decomposition rules and then stating our results.  Then in \S\ref{sec:recrelations} we determine the recurrence relations for the sequences, in \S\ref{sec:growthratefibquilt} we prove our claims on the growth of the average number of decompositions from the Fibonacci Quilt sequence, and then analyze the greedy algorithm and a generalization (for the Fibonacci Quilt) in \S\ref{sec:greedy}.

\subsection{The $(s,b)$-Generacci Sequence and the Fibonacci Quilt Sequence}\label{defofseqs}

\subsubsection{The $(s,b)$-Generacci Sequence} \ \\


One interpretation of Zeckendorf's Theorem \cite{Ze} is that the Fibonacci sequence is the unique sequence from which all natural numbers can be expressed as a sum of nonconsecutive terms.  Note there are two ingredients to the rendition: a sequence and a rule for determining what is a legal decomposition. An equivalent formulation for the Fibonacci numbers is to consider the sequence divided into bins of size one and decompositions can use the element in a bin at most once and cannot use elements from adjacent bins. A generalization of this bin idea was explored by the authors in \cite{CFHMN1}, where  bins of size 2 with the same non-adjacency condition were considered;  the sequence that arose was called the Kentucky sequence. The Kentucky sequence is what we now refer to here  as the $(1,2)$-Generacci sequence. This leads to a natural extension where we consider bins of size $b$ and any two summands of a decomposition must come from distinct bins with at least $s$ bins between them. We now give the technical definitions of the $(s,b)$-Generacci sequences and their associated legal decompositions.



\begin{defi}[$(s,b)$-Generacci legal decompositions]\label{def:sb}
For fixed integers $s, b \geq 1$, let an increasing sequence of positive integers $\{a_i\}_{i=1}^\infty$ and a family of subsequences $\mathcal{B}_n=\{a_{b(n-1)+1},\ldots,a_{bn}\}$ be given (we call these subsequences {\em bins}). We declare a decomposition of an integer $m = a_{\ell_1} + a_{\ell_2} + \dots + a_{\ell_k}$  where $a_{\ell_i} > a_{\ell_{i+1}}$  to be an {\em $(s,b)$-Generacci  legal decomposition} provided $\{a_{\ell_i}, a_{\ell_{i+1}}\} \not\subset \mathcal{B}_{j-s}\cup \mathcal{B}_{j-s+1}\cup \dots \cup \mathcal{B}_{j}$ for all $i,j$. (We say $\mathcal{B}_{j} = \emptyset$ for $j\leq 0$.)
\end{defi}

Thus if we have a summand $a_{\ell_i} \in \mathcal{B}_j$ in a legal decomposition, we cannot have any other summands from that bin, nor any summands from any of the $s$ bins preceding or any of the $s$ bins following $\mathcal{B}_j$.


\begin{defi}[$(s,b)$-Generacci sequence]\label{sbDefi}
For fixed integers $s, b \geq 1$, an increasing sequence of positive integers $\{a_i\}_{i=1}^\infty$ is the {\em $(s,b)$-Generacci sequence} if every $a_i$ for $i \geq 1$ is the smallest positive integer that does not have an $(s,b)$-Generacci legal decomposition using the elements $\{a_1, \dots, a_{i-1}\}.$
\end{defi}

Using the above definition and Zeckendorf's theorem,   we see that the $(1,1)$-Generacci sequence is the Fibonacci sequence (appropriately normalized).  Some other known sequences arising  from the  $(s,b)$-Generacci sequences are Narayana's cow sequence, which is the $(2,1)$-Generacci sequence, and the Kentucky sequence, which is the $(1,2)$-Generacci sequence.

\begin{theorem}[Recurrence Relation and Explicit Formula]\label{thrm:recurrence} Let $s,b\geq 1$ be fixed. If $n>(s+1)b+1$, then the $n$\textsuperscript{th} term of the $(s,b)$-Generacci sequence is given by the recurrence relation
\begin{align}\label{recurrence}
a_n &\ =\ a_{n-b}+ba_{n-(s+1)b}.
\end{align}
We have a generalized Binet's formula, with
\begin{align}
a_n & \ = \  c_1\lambda_1^n \left[1 + O\left((\lambda_2/\lambda_1)^n\right)\right]
\label{explicitConstant}\end{align}
where $\lambda_1$ is the largest root of $x^{(s+1)b} - x^{sb} - b = 0,$ and $c_1$ and $\lambda_2$ are constants with $\lambda_1 > 1,$ $c_1 > 0 $ and $|\lambda_2| < \lambda_1$.
\end{theorem}

\begin{rek} The $(s,b)$-Generacci sequence also satisfies the recurrence
\begin{align}\label{f-recurrence}
a_n &\ =\ a_{n-1}+a_{n-1 - f(n-1)},
\end{align}
where $f(kb+j) = sb+j-1$ for  $j=1, \dots, b$. While this representation does have its leading coefficient positive, note the depth $L=f(n-1)+1$ is \emph{not} independent of $n$, and thus this representation is not a Positive Linear Recurrence.
\end{rek}

%

The proof of Theorem \ref{thrm:recurrence} is given in \S\ref{sec:proofofrecurrence}.  We note that the leading term in the recurrence in  \eqref{recurrence}  is zero whenever $b \ge 2$, and hence this sequence falls out of the scope of the Positive Linear Recurrences results.



\subsubsection{Fibonacci Quilt Sequence} \ \\


The Fibonacci Quilt sequence arose from the goal of finding a sequence coming from a 2-dimensional process. We begin by recalling the beautiful fact that the Fibonacci numbers tile the plane with squares spiraling to infinity, where the side length of the $n$\textsuperscript{th} square is $F_n$ (see Figure \ref{fig:fib_spiral}; note that here we start the Fibonacci sequence with two 1's).
\begin{figure}[h]
\begin{center}
\scalebox{.4}{\includegraphics{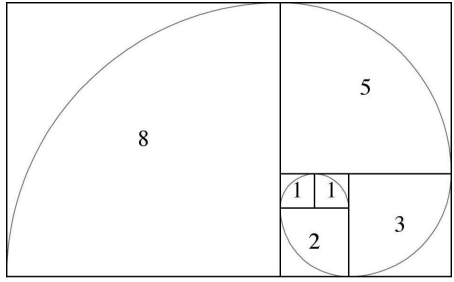}}
\caption{\label{fig:fib_spiral} The (start of the) Fibonacci Spiral.}
\end{center}\end{figure}

Inspired by Zeckendorf decomposition rules and by the Fibonacci spiral we define the following notion of legal decompositions and create the associated integer sequence which we call the Fibonacci Quilt sequence. The spiral depicted in Figure \ref{fig:fib_spiral} can be viewed as a log cabin \emph{quilt} pattern, such as that presented in Figure \ref{fig:fibquiltspiralseq} (left). Hence we adopt the name Fibonacci Quilt sequence.

\begin{definition}[FQ-legal decomposition]
Let an increasing sequence of positive integers $\{q_i\}_{i=1}^\infty$ be given. We declare a decomposition of an integer
\be m \ = \  q_{\ell_1} + q_{\ell_2} +\cdots + q_{\ell_t}\ee (where $q_{\ell_i} > q_{\ell_{i+1}}$)
to be an {\em FQ-legal decomposition} if for all $i,j$, $|\ell_i-\ell_j|\neq 0,1,3,4$ and $\{1,3\}\not\subset\{\ell_1,\ell_2,\ldots \ell_t\}$.
\end{definition}

This means that if the terms of the sequence are arranged in a spiral in the rectangles of a log cabin quilt, we cannot use two terms if they share part of an edge.  Figure \ref{fig:fibquiltspiralseq} shows that  $q_n+q_{n-1}$ is not legal, but $q_n+q_{n-2}$ is legal for $n \ge 4$. The starting pattern of the quilt forbids decompositions that contain $q_3+q_1$.


We define a new sequence $\{q_n\}$, called the Fibonacci Quilt sequence, in the following way.

\begin{definition}[Fibonacci Quilt sequence]
An increasing sequence of positive integers $\{q_i\}_{i=1}^\infty$ is called the {\em Fibonacci Quilt sequence}
if every $q_i$ ($i\geq1$) is the smallest positive integer that does not have an FQ-legal decomposition
using the elements $\{q_1,\ldots,q_{i-1}\}$.
\end{definition}

From the definition of an FQ-legal  decomposition, the reader can see that the first five terms of the sequence must be $\{1, 2, 3, 4, 5\}$. We have $q_6\neq 6$ as $6 = q_4 +q_2 = 4+2$ is an FQ-legal decomposition. We must have $q_6 = 7$. Continuing we have the start of the Fibonacci Quilt sequence displayed in Figure \ref{fig:fibquiltspiralseq} (right). Note that with the exception of a few initial terms, the Fibonacci Quilt sequence and the Padovan (see entry A000931 from the OEIS) sequence are eventually identical.

\begin{figure}
\begin{minipage}{\textwidth}
\begin{minipage}[b]{0.5\textwidth}
\centering
\resizebox{.85\textwidth}{!}{
\begin{tikzpicture}
\draw (0,0) rectangle (1,1);
\node at (.5,.5) {$q_1$};
\draw (0,0) rectangle (1,-1);
\node at (.5,-.5) {$q_2$};
\draw (1,-1) rectangle (2,1);
\node at (1.5,0) {$q_3$};
\draw (2,1) rectangle (0,2);
\node at (1,1.5) {$q_4$};
\draw (0,2) rectangle (-1,-1);
\node at (-.5,.5) {$q_5$};
\draw (-1,-1) rectangle (2,-2);
\node at (.5,-1.5) {$q_6$};
\draw (2,-2) rectangle (3,2);
\node at (2.5,0) {$q_7$};
\draw (3,2) rectangle (-1,3);
\node at (1,2.5) {$\vdots$};
\draw (-1,3) rectangle (-2,-2);
\node at (-1.5,.5) {$\cdots$};
\draw (-2,-2) rectangle (3,-3);
\node at (.5,-2.5) {$\vdots$};
\draw (3,-3) rectangle (4,3);
\node at (3.5,0) {$\cdots$};
\draw (4,3) rectangle (-2,4);
\node at (1,3.5) {$q_{n-4}$};
\draw (-2,4) rectangle (-3,-3);
\node at (-2.5,.5) {$q_{n-3}$};
\draw (-3,-3) rectangle (4,-4);
\node at (.5,-3.5) {$q_{n-2}$};
\draw (4,-4) rectangle (5,4);
\node at (4.5,0) {$q_{n-1}$};
\draw (5,4) rectangle (-3,5);
\node at (1,4.5) {$q_{n}$};
\draw (-3,5) rectangle (-4,-4);
\node at (-3.5,.5) {$q_{n+1}$};
\draw (-4,-4) rectangle (5,-5);
\node at (.5,-4.5) {$q_{n+2}$};
\draw (5,-5) rectangle (6,5);
\node at (5.5,0) {$q_{n+3}$};
\end{tikzpicture}}
 \end{minipage}
  \begin{minipage}[b]{0.5\textwidth}
\centering
\resizebox{.85\textwidth}{!}{
\begin{tikzpicture}
\draw (0,0) rectangle (1,1);
\node at (.5,.5) {$1$};
\draw (0,0) rectangle (1,-1);
\node at (.5,-.5) {$2$};
\draw (1,-1) rectangle (2,1);
\node at (1.5,0) {$3$};
\draw (2,1) rectangle (0,2);
\node at (1,1.5) {$4$};
\draw (0,2) rectangle (-1,-1);
\node at (-.5,.5) {$5$};
\draw (-1,-1) rectangle (2,-2);
\node at (.5,-1.5) {$7$};
\draw (2,-2) rectangle (3,2);
\node at (2.5,0) {$9$};
\draw (3,2) rectangle (-1,3);
\node at (1,2.5) {$12$};
\draw (-1,3) rectangle (-2,-2);
\node at (-1.5,.5) {$16$};
\draw (-2,-2) rectangle (3,-3);
\node at (.5,-2.5) {$21$};
\draw (3,-3) rectangle (4,3);
\node at (3.5,0) {$28$};
\draw (4,3) rectangle (-2,4);
\node at (1,3.5) {$37$};
\draw (-2,4) rectangle (-3,-3);
\node at (-2.5,.5) {$49$};
\draw (-3,-3) rectangle (4,-4);
\node at (.5,-3.5) {$65$};
\draw (4,-4) rectangle (5,4);
\node at (4.5,0) {$86$};
\draw (5,4) rectangle (-3,5);
\node at (1,4.5) {$114$};
\draw (-3,5) rectangle (-4,-4);
\node at (-3.5,.5) {$151$};
\draw (-4,-4) rectangle (5,-5);
\node at (.5,-4.5) {$200$};
\draw (5,-5) rectangle (6,5);
\node at (5.5,0) {$265$};
\draw (6,5) rectangle (-4,6);
\node at (1,5.5) {$351$};
\draw (-4,6) rectangle (-5,-5);
\node at (-4.5,.5) {$465$};
\end{tikzpicture}}
\end{minipage}
  \end{minipage}
\caption{\label{fig:fibquiltspiralseq} (Left) Log Cabin Quilt Pattern. (Right) First  few terms of the Fibonacci Quilt sequence.}
\end{figure}
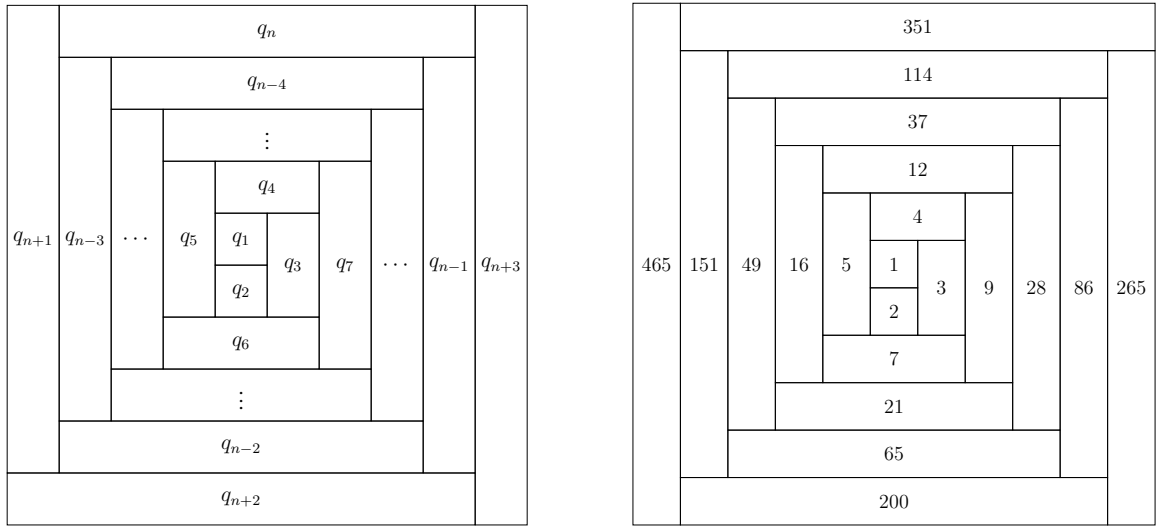

\begin{thm}[Recurrence Relations]\label{thm:rrfibquilt} Let $q_n$ denote the $n$\textsuperscript{th} term in the Fibonacci Quilt. Then
 \be \mbox{{\rm for} }\, n\geq 6,\,  q_{n+1}\ =\ q_{n}  + q_{n-4},\label{eq:rr1}\ee
\be \mbox{{\rm for} } \, n \geq 5, \, q_{n+1}\ =\ q_{n-1}  + q_{n-2}, \label{eq:rr2}\ee
\be \sum_{i=1}^n q_i\ =\ q_{n+5} -6.   \label{eq:sum} \ee
\end{thm}

The proof is given in \S\ref{sec:proofofreccurencefq}.

\begin{rek} At first the above theorem seems to suggest that the Fibonacci Quilt \emph{is} a PLR, as \eqref{eq:rr1} gives us a recurrence where the leading coefficient is positive and, unlike the alternative expression for the $(s,b)$-Generacci, this time the depth is fixed. The reason it is not a PLR is subtle, and has to do with the second part of the definition: the decomposition law. The decomposition law is \emph{not} from using \eqref{eq:rr1} to reduce summands, but from the geometry of the spiral. It is worth remarking that \eqref{eq:rr2} is the minimal length recurrence for this sequence, and the characteristic polynomial arising from \eqref{eq:rr1} is divisible by the polynomial from \eqref{eq:rr2}.
\end{rek}


\subsection{Results}


Our theorems are for two sequences which satisfy recurrences with leading term zero. Prior results in the literature mostly considered Positive Linear Recurrences and results included the uniqueness of legal decompositions, Gaussian behavior of the number of summands, and exponential decay in the distribution of the gaps between summands \cite{BBGILMT,BILMT,DDKMMV,DDKMV,MW1,MW2}. In \cite{CFHMN1}, a first example of a non-positive linear recurrence appeared and the aforementioned results were proved using arguments technically similar to those already present in the literature. What is new in this paper are two extensions of the work presented in \cite{CFHMN1}. The first is the $(s,b)$-Generacci sequence, whose legal decompositions are unique but where new techniques are required to prove its various properties. The second is the more interesting newly discovered Fibonacci Quilt sequence, which displays drastically different behavior, one consequence being that the FQ-legal decompositions are not unique (for example, there are three distinct FQ-legal decompositions of 106: 86+16+4, 86+12+7+1, and 65+37+4).


\subsubsection{Decomposition results}

\begin{theorem}[Uniqueness of Decompositions for $(s,b)$-Generacci]\label{uniquedecomp}
For each pair of  integers $s,b \geq 1$,  a unique $(s,b)$-Generacci sequence exists.  Consequently,  for a  given pair of  integers $s,b \geq 1$,  every positive integer can be written uniquely as a sum of distinct terms of the $(s,b)$-Generacci sequence where no two summands are in the same bin, and between any two summands there are at least $s$ bins between them.
\end{theorem}

As Theorem~\ref{uniquedecomp} follows from a similar argument to that in the appendix of \cite{CFHMN1}, we omit it in this paper.

\begin{rek}
We could also prove this result by showing that our sequence and legal decomposition rule give rise to an $f$-decomposition. These were defined and studied in \cite{DDKMMV}, and briefly a valid $f$-decomposition means that for each summand chosen a block of consecutive summands before are not available for use, and that number depends solely on $n$. The methods of \cite{DDKMMV} are applicable and yield that each positive integer has a unique legal decomposition.

These results are not available for the Fibonacci Quilt sequence, as the FQ-legal decomposition is not an $f$-decomposition. The reason is that in an $f$-decomposition there is a function $f$ such that if we have $q_n$ then we cannot have any of the $f(n)$ terms of the sequence immediately prior to $q_n$. There is no such $f$ for the Fibonacci Quilt sequence, as for $n \ge 8$ if we have $q_n$ we cannot have $q_{n-1}$ and $q_{n-3}$ but we can have $q_{n-2}$.
\end{rek}


We have already seen that the Fibonacci Quilt leads to non-unique decompositions; this is just the beginning of the difference in behavior. The first result concerns the exponential number of FQ-legal decompositions as we decompose larger integers. First we need to introduce some notation.
Let $\{q_n\}$ denote the Fibonacci Quilt  sequence. For each positive integer $m$ let $\dfq(m)$ denote the number of FQ-legal decomposition of $m$, and $\dave(n)$ the average number of FQ-legal decompositions of integers in $I_n := [0,q_{n+1})$; thus \be \dave(n) \ := \ \frac1{q_{n+1}} \sum_{m=0}^{q_{n+1}-1} \dfq(m). \ee In \S\ref{sec:growthratefibquilt} we prove the following.

\begin{thm}[Growth Rate of Average Number of Decompositions]\label{thm:growthratenumberdecomp} Let $r_1$ be the largest root of $r^7 - r^6 - r^2  -1 = 0$ (so $r_1 \approx 1.39704$)  and let $\lambda_1$ be the largest root of $x^3 - x - 1 = 0$ (so $\lambda_1 = \frac13 \left(\frac{27}{2} - \frac{3 \sqrt{69}}{2}\right)^{1/3} + 3^{-2/3}\left(\frac12 \left(9 + \sqrt{69}\right)\right)^{1/3} \approx 1.32472$), and set $\lambda = r_1/\lambda_1 \approx 1.05459$. There exist computable constants $C_2 > C_1 > 0$ such that for all $n$ sufficiently large, \be C_1 \lambda^n \ \le \ \dave(n)\ \le\ C_2 \lambda^n.\ee Thus the average number of FQ-legal decompositions of integers in $[0, q_{n+1})$ tends to infinity exponentially fast.
\end{thm}

\begin{rek} At the cost of additional algebra one could prove the existence of a constant $C$ such that $\dave(n) \sim C \lambda^n$; however, as the interesting part of the above theorem is the exponential growth and not the multiplicative factor, we prefer to give the simpler proof which captures the correct growth rate. 
\end{rek}


We end with another new behavior. For many of the previous recurrences, the greedy algorithm successfully terminates in a legal decomposition; that is \emph{not} the case for the Fibonacci Quilt sequence. In \S\ref{sec:greedy} we prove the following.

\begin{thm}\label{thm:successgreedyalg} There is a computable constant $\rho \in (0,1)$ such that, as $n\to\infty$, the percentage of positive integers in $[1,q_n)$ where the greedy algorithm terminates in a Fibonacci Quilt legal decomposition converges to $\rho$. This constant is  approximately .92627.
\end{thm}

Interestingly, a simple modification of the greedy algorithm \emph{does} always terminate in a legal decomposition, and this decomposition yields  a minimal number of summands.

\begin{definition}[Greedy-6 Decomposition]\label{alg:greedy6} The Greedy-6 Decomposition writes $m$ as a sum of Fibonacci Quilt numbers as follows:
\begin{itemize}
\item if there is an $n$ with $m=q_n$ then we are done,
\item if $m = 6$ then we decompose $m$ as $q_4+q_2$ and we are done, and
\item if $m \geq q_6$ and $m \neq q_n$ for all $n \geq 1$, then we write $m=q_{\ell_1} + x$ where $q_{\ell_1}< m< q_{\ell_1+1}$ and $x > 0$, and then iterate the process with input $m :=x$.
\end{itemize}
We denote the decomposition that results from the Greedy-6 Algorithm by $\mathcal{G}(m)$.
\end{definition}

\begin{theorem}\label{thm:greedy6}
For all $m > 0,$ the Greedy-6 Algorithm results in a FQ-legal decomposition.  Moreover, if $\mathcal{G}(m) = q_{\ell_1}+q_{\ell_2}+\dots +q_{\ell_{t-1}}+q_{\ell_t}$ with $q_{\ell_1}>q_{\ell_2}>\cdots>q_{\ell_t}$,  then the decomposition satisfies exactly one of the following conditions:
\begin{enumerate}
\item $\ell_i-\ell_{i+1}\geq 5$ for all $i$ or
\item $\ell_i-\ell_{i+1}\geq 5$ for $i\leq t-3$ and $\ell_{t-2}\geq10,\, \ell_{t-1}=4, \,\ell_{t}=2$.
\end{enumerate}

Further, if $m=q_{\ell_1}+q_{\ell_2}+\cdots+q_{\ell_{t-1}}+q_{\ell_t}$ with $q_{\ell_1}>q_{\ell_2}>\cdots>q_{\ell_t}$  denotes a decomposition of $m$ where either
\begin{enumerate}
\item $\ell_i-\ell_{i+1}\geq 5$ for all $i$ or
\item $\ell_i-\ell_{i+1}\geq 5$ for $i\leq t-3$ and $\ell_{t-2}\geq10,\, \ell_{t-1}=4, \,\ell_{t}=2$,
\end{enumerate}
then $q_{\ell_1}+q_{\ell_2}+\cdots+q_{\ell_{t-1}}+q_{\ell_t}=\mathcal{G}(m)$. That is, the decomposition of $m$ is the Greedy-6 decomposition.
\end{theorem}

Let $\mathcal{D}(m)$ be a given decomposition of $m$ as a sum of Fibonacci Quilt numbers (not necessarily legal): \be m \ = \ c_1 q_1 + c_2 q_2 + \cdots + c_n q_n, \ \ \ c_i \in \{0, 1, 2, \dots\}. \ee
We define 
the number of summands by \be \#{\rm summands}(\mathcal{D}(m)) \ := \ c_1 + c_2 + \cdots + c_n. \ee

\begin{thm}\label{thm:Dm} If $\mathcal{D}(m)$ is any decomposition of $m$ as a sum of Fibonacci Quilt numbers, then \be \#{\rm summands}(\mathcal{G}(m)) \ \le \ \#{\rm summands}(\mathcal{D}(m)).\ee
\end{thm}

\subsubsection{Gaussian Behavior of Number of Summands in $(s,b)$-Generacci  legal decompositions}

Below we report on the distribution of the number of summands in the $(s,b)$-Generacci legal decompositions. In attacking this problem we developed a new technique similar to ones used before but critically different in that we are able to bypass technical assumptions that other papers needed to prove a Gaussian distribution. We elaborate on this method in \cite{CFHMNPX}, where we also determine the distribution of gaps between summands. We have chosen to concentrate on the Fibonacci Quilt results in this paper, and just state many of the $(s,b)$-Generacci outcomes, as we see the same behavior as in other systems for the $(s,b)$-Generacci numbers, but see fundamentally new behavior for the Fibonacci Quilt sequence.

\begin{theorem}[Gaussian Behavior of Summands for $(s,b)$-Generacci]\label{thm:gaussian}
Let the random variable $Y_n$ denote the number of summands in the (unique) $(s,b)$-Generacci legal decomposition of an integer picked at random from $[0, a_{bn+1})$ with uniform probability.\footnote{Using the methods of \cite{BDEMMTTW}, these results can be extended to hold almost surely for sufficiently large sub-interval of $[a_{(n-1)b+1}, a_{bn+1})$.}
Then for $\mu_n$ and $\sigma_n^2$, the mean and variance of $Y_n$, we have
\begin{align} \mu_n & \ = \   An+B+o(1)\label{muConstantA} \\
\sigma_n^2 & \ = \   Cn+D+o(1)
 \end{align} for some positive constants $A,B,C,D$.
Moreover if we
normalize $Y_n$ to $Y_n' = (Y_n - \mu_n)/\sigma_n$,  then $Y_n'$ converges in distribution to the standard normal distribution as $n \rightarrow \infty$.
\end{theorem}

Unfortunately, the above methods do not directly generalize to Gaussian results for the Fibonacci Quilt sequence. Interestingly and fortunately there is a strong connection between the two sequences, and in \cite{CFHMNPX} we show how to interpret many questions concerning the Fibonacci Quilt sequence to a weighted average of several copies of the $(4,1)$-Generacci sequence. This correspondence is not available for questions on unique decomposition, but does immediately yield Gaussian behavior and determines the limiting behavior of the individual gap measures.



\section{Recurrence Relations}\label{sec:recrelations}

\subsection{Recurrence Relations for the $(s,b)$-Generacci Sequence}\label{sec:proofofrecurrence}

Recall  that for  $s,b\geq 1$, an $(s,b)$-Generacci decomposition of a positive integer is legal if the following conditions hold.
\begin{enumerate}
\item No term $a_i$ is used more than once.
\item No two distinct terms $a_i$, $a_j$ in a decomposition can have indices $i$, $j$ from the same bin.
\item If $a_i$ and $a_j$ are summands in a legal decomposition, then there are at least $s$ bins between them.
\end{enumerate}

The terms of the $(s,b)$-Generacci sequence can be pictured as follows:
\be \underbracket{a_1,\ldots,a_{b}}_{\mathcal{B}_1}\ ,\underbracket{a_{1+b},\ldots,a_{2b}}_{\mathcal{B}_2}\ ,\ldots,\ \underbracket{a_{1+nb},\ldots,a_{(n+1)b}}_{\mathcal{B}_{n+1}}\ ,\underbracket{a_{1+(n+1)b},\ldots,a_{(n+2)b}}_{\mathcal{B}_{n+2}} \ , \underbracket{a_{1+(n+2)b},\ldots,a_{(n+3)b}}_{\mathcal{B}_{n+3}}\ ,\ldots.\ee


We now prove the following results related to the elements of the $(s,b)$-Generacci sequence.

\begin{lemma}\label{l1}
If $s,b\geq 1$, then $a_i=i$ for all $1\leq i\leq (s+1)b+1$, where $a_i$ is the $i$\textsuperscript{{\rm th}} term in the $(s,b)$-Generacci sequence.
\end{lemma}

\begin{proof}This follows directly from the definition of the $(s,b)$-Generacci sequence.
That is, we  note that at the $(s+1)$\textsuperscript{th}-bin, we clearly have $s$-many bins to the left, yet we are unable to use any elements from those bins to decompose any new integers. Thus $a_{i}=i$, for all $1\leq i\leq (s+1)b+1$.
\end{proof}



\begin{lemma}
If $k$ can be decomposed using summands $\{a_1,\ldots,a_p\}$, then so can $k-1$.
\end{lemma}

\begin{proof}
Let $k=a_{\ell_1}+a_{\ell_2}+\cdots+a_{\ell_t}$ with $\ell_1>\ell_2>\cdots>\ell_t$ be a legal decomposition of $k$. So $k-1=a_{\ell_1}+a_{\ell_2}+\cdots+(a_{\ell_t}-1)$.

It must be the case that either $a_{\ell_t}-1$ is zero or it has a legal decomposition with summands indexed smaller than $\ell_t$, as  $a_{\ell_t}$ was added because it was the smallest integer that could not be legally decomposed with summands indexed smaller than $\ell_t$. If $\ell_t$ was sufficiently distant from $\ell_{t-1}$ for the decomposition of $k$ to be legal, using summands with even smaller indices does not create an illegal interaction with the remaining summands $a_{\ell_1},\ldots,a_{\ell_{t-1}}$.
\end{proof}

This lemma allows us to conclude that the smallest integer that does not have a legal decomposition using $\{a_1,\ldots,a_n\}$ is one more than the largest integer that does have a legal decomposition using $\{a_1,\ldots,a_n\}$.

\begin{lemma}\label{jthinbinRR}If $s,b,n\geq 1$ and $1\leq j\leq b+1$, then
\begin{align}
a_{j+nb}&\ = \ a_{1+nb}+(j-1)a_{1+(n-s)b}.\label{inbinRR}
\end{align}
\end{lemma}

\begin{proof}
The term $a_{1+nb}$ is the first entry in the $(n+1)$\textsuperscript{st} bin and trivially satisfies the recursion relation for $j=1$.

Recall a legal decomposition containing a member of the $(n+1)$\textsuperscript{st} bin would  not have other addends from any of bins $\{\mathcal{B}_{n-s+1}, \mathcal{B}_{n-s+2}, \dots , \mathcal{B}_{n}, \mathcal{B}_{n+1}\} $. So by construction we have $a_{2 + nb} = a_{1+nb} +a_{1+(n-s)b}$,  as the largest integer that can be legally decomposed using addends from bins $\mathcal{B}_1, \mathcal{B}_2, \dots, \mathcal{B}_{n-s}$ is $a_{1+(n-s)b} - 1.$

Using the same argument we have
\begin{align}
a_{3+nb}&  \ = \ a_{2+nb} + a_{1+(n-s)b} \ = \  a_{1+nb} + 2a_{1+(n-s)b}.
\end{align}
We proceed similarly for $j = 4, \dots, b.$
For $j=b+1$, the term $a_{b+1 + nb} = a_{1+(n+1)b}$ is the first entry in the $(n+2)$\textsuperscript{nd} bin.  By construction $a_{1+(n+1)b} = a_{(n+1)b} + a_{1+(n-s)b}$.  Using Equation \eqref{inbinRR} with $j=b$ we have \be a_{1+(n+1)b}\ =\ a_{(n+1)b} + a_{1+(n-s)b} = a_{1+nb}+(b-1)a_{1+(n-s)b} + a_{1+(n-s)b}\ =\ a_{1+nb}+ba_{1+(n-s)b}.\ee   \end{proof}

\begin{proof}[Proof of Theorem \ref{thrm:recurrence}]
Fix $s,b\geq 1$.  We  proceed by considering $i$ of the form $j+nb$, $j \in \{1, \ldots, b\}$, so $a_i = a_{j+nb}$ is the $j$\textsuperscript{th} entry in the $(n+1)$\textsuperscript{st} bin.  Using Lemma ~\ref{jthinbinRR},
%
\begin{align}
a_{j+nb}&\ = \ a_{1+nb}+(j-1)a_{1+(n-s)b} \nonumber\\
&\ = \ a_{1+(n-1)b}+ba_{1+(n-s-1)b}+(j-1)a_{1+(n-s)b}\nonumber\\
&\ = \  a_{1+(n-1)b}+(j-1)a_{1+(n-s-1)b}+(b-j+1)a_{1+(n-s-1)b}+(j-1)a_{1+(n-s)b}\nonumber\\
&\ = \  a_{j+(n-1)b}+(b-j+1)a_{1+(n-s-1)b}+(j-1)a_{1+(n-s)b}.
\end{align}

Again using the construction of our sequence we have $a_{1+(n-s)b} = a_{(n-s)b} + a_{1+(n-2s-1)b}$.  This substitution gives
\begin{align}
a_{j+nb}&\ = \ a_{j+(n-1)b}+(b-j+1)a_{1+(n-s-1)b}+(j-1)a_{(n-s)b} + (j-1)a_{1+(n-2s-1)b}\nonumber\\
&\ = \  a_{j+(n-1)b}+a_{1+(n-s-1)b}+ (j-1)a_{1+(n-2s-1)b}+(b-j)a_{1+(n-s-1)b}+(j-1)a_{(n-s)b} \nonumber\\
&\ = \  a_{j+(n-1)b}+a_{j+(n-s-1)b}+(b-j)a_{1+(n-s-1)b}+(j-1)a_{(n-s)b} .\label{RRalmostdone}
\end{align}

Note that by Lemma \ref{jthinbinRR},  $a_{(n-s)b} = a_{1+(n-s-1)b} + (b-1)a_{1+(n-2s-1)b}$, so the last two terms in  \eqref{RRalmostdone} may be simplified as
\begin{align}
 (b-j)a_{1+(n-s-1)b}+(j-1)a_{1+(n-s-1)b} + (j-1)(b-1)a_{1+(n-2s-1)b}  & \nonumber\\
 \ = \    (b-1)\left[a_{1+(n-s-1)b}+(j-1)a_{1+(n-2s-1)b} \right]  & \nonumber\\
 \ = \   (b-1)a_{j+(n-s-1)b}. & \label{RRtada}
\end{align}
Substituting \eqref{RRtada} into Equation \eqref{RRalmostdone} yields
\begin{align}
a_{j+nb}&\ = \ a_{j+(n-1)b}+a_{j+(n-s-1)b}+(b-1)a_{j+(n-s-1)b}\nonumber\\
&\ = \ a_{j+(n-1)b}+ba_{j+(n-s-1)b},
\end{align}
which completes the proof of  the first part of Theorem \ref{thrm:recurrence}.

For the proof of  the second part,  we have from  Lemma \ref{jthinbinRR}
\begin{align}
 a_{j+nb}  & \ = \  a _{j-1+nb}+ a_{1+(n-s)b},   \end{align}
thus
\begin{align}
 a_{j+nb}  & \ = \   a _{j-1+nb}+ a_{j-1+nb - (sb + j -2)}
 \end{align}
for $j=2, \ldots, b+1$.  The result now follows  if  we define $f(j+nb) = sb+j-1,$  for $ j=1, \ldots, b$.

We prove the Generalized Binet Formula and the approximation in Appendix \ref{sec:genbinetformSBGen}.
\end{proof}

\subsection{Recurrence Relations for Fibonacci Quilt Sequence}\label{sec:proofofreccurencefq}

\begin{proof}[Proof of Theorem \ref{thm:rrfibquilt}]
The proof is by induction. The basis cases for $n\leq11$ can be checked by brute force.

By construction, we can legally decompose all numbers in the interval $[1, q_{n-4}-1]$ using terms in  $\{q_1,\ldots,q_{n-5}\}$; $q_{n-4}$ was added to the sequence because it was the first number that could not be decomposed using those terms. So, using $q_n$, we can legally decompose all numbers in the interval, $[q_n,q_n+ q_{n-4}-1]$. In fact, we can decompose all numbers in the interval $[1,q_n+ q_{n-4}-1]$ using $\{q_1,\ldots,q_{n}\}$. The term $q_{n+1}$ will be  the smallest number that we cannot legally decompose using $\{q_1,\ldots,q_{n}\}$.   The argument above shows that $q_{n+1}\geq q_n+ q_{n-4}$.

Notice
\begin{eqnarray}\label{eq:pf}  q_{n}+q_{n-4}
 & \ =\  &  (q_{n-1}+q_{n-5})+q_{n-4} \nonumber\\
 & \ = \  &  q_{n-1}+(q_{n-4}+q_{n-5}) \nonumber\\
 & \ = \  & q_{n-1}+q_{n-2}. \end{eqnarray}

It remains to show that there is no legal decomposition of $m=q_n+ q_{n-4}=q_{n-1}+q_{n-2}$. If $q_n$ were in the decomposition of $m$, the remaining summands would have to add to $q_{n-4}$. But that is a contradiction as $q_{n-4}$ was added to the sequence because it had no legal decompositions as sums of other terms. Similarly, we can see that
 any legal decomposition of $m$ does not use $ q_{n-1},q_{n-2},q_{n-4}$.

Notice that $q_{n-3}$ must be part of any possible legal decomposition of $m$: if it were not, then $m < \sum_{i=1}^{n-5} q_i = q_{n} -6 < q_n < q_n + q_{n-4} =m$.   Hence any legal decomposition would have $m = q_{n-3} + x$, where the largest possible summand in the decomposition of $x$ is $q_{n-5}$.

Now assume we have a legal decomposition of $m = q_{n-3} + x$. There are two cases.\\ \

\noindent {\bf Case 1}: The legal decomposition of $x$ uses $q_{n-5}$ as a summand. So
\be m \ = \  q_{n-3} + x \ = \  q_{n-3} + q_{n-5} + y\ee
 and $y$ can be legally decomposed using summands from $\{q_1, q_2, \dots , q_{n-10}\}$. Then using Equation \eqref{eq:sum}, $y < \sum_{i=1}^{n-10} q_i= q_{n-5} -6.$ This leads us to the following:
\bea
q_n+q_{n-4}\,=\,m &\ < \ & q_{n-3} + q_{n-5} + q_{n-5} -6\nonumber\\
&<& q_{n-3} + q_{n-4} + q_{n-5} -6\nonumber\\
&=& q_{n-1} + q_{n-5} -6\nonumber\\
&=& q_{n}-6\nonumber\\
&<& q_n,
\eea a contradiction.\\ \

\noindent {\bf Case 2}: The largest possible summand used in the legal decomposition of $x$ is $q_{n-8}$. Thus
 \bea
q_n+q_{n-4}\,=\,m &\ < \ & q_{n-3} + \sum_{i=1}^{n-8} q_i\nonumber\\
&=& q_{n-3} +  q_{n-3} -6\nonumber\\
&<& q_{n-2} + q_{n-3}\nonumber\\
&<& q_n,
\eea another contradiction.

So $m$ cannot be legally decomposed using $\{q_1,\ldots,q_{n}\}$ and $q_{n+1} = q_n + q_{n-4}. $
The proof of Equation \eqref{eq:rr2} follows from the work done in Equation \eqref{eq:pf}.
To prove Equation \eqref{eq:sum}, note
\be  \sum_{i=1}^{n+1}q_i\ = \  q_{n+1}+\sum_{i=1}^{n}q_i  \ = \  q_{n+1}+q_{n+5}-6\ = \ q_{n+6}-6.\qedhere \ee
\end{proof}

\begin{prop}[Explicit Formula]
\label{prop:fibquiltnumbers} Let $q_n$ denote the $n$\textsuperscript{th} term in the Fibonacci Quilt sequence. Then \be q_n \ = \ \alpha_1 \lambda_1^n + \alpha_2 \lambda_2^n + \alpha_3 \overline{\lambda_2}^n, \ee where $\alpha_1 \approx 1.26724 $, \be \lambda_1 \ = \ \frac13 \left(\frac{27}{2} - \frac{3 \sqrt{69}}{2}\right)^{1/3} + \frac{\left(\frac12 \left(9 + \sqrt{69}\right)\right)^{1/3}}{3^{2/3}} \ \approx \ 1.32472 \ee and $\lambda_2 \approx -0.662359 - 0.56228i$ (which has absolute value approximately 0.8688).
\end{prop}

\begin{proof} Using the recurrence relation from Equation \eqref{eq:rr1} in Theorem \ref{thm:rrfibquilt}, we have the characteristic equation
\begin{align}
x^3 \ = \  x+1.
\end{align}
Hence $q_n \ = \ \alpha_1 \lambda_1^n + \alpha_2 \lambda_2^n + \alpha_3 \overline{\lambda_2}^n$, where $\lambda_1$,  $\lambda_2$ and $\overline{\lambda_2}$ are the three distinct solutions to the characteristic equation, which are easily found by the cubic formula.

We solve for the $\alpha_i$ using the first few terms of the sequence. Straightforward calculations reveal
\begin{align}
\alpha_1&\ \approx\ \ 1.26724\nonumber\\
\alpha_2&\ \approx\ -0.13362 + 0.128277 i\nonumber\\
\alpha_3&\ \approx\ -0.13362 - 0.128277 i,
\end{align} completing the proof.
\end{proof}


\section{Growth Rate of Number of Decompositions for Fibonacci Quilt Sequence}\label{sec:growthratefibquilt}

We prove Theorem \ref{thm:growthratenumberdecomp} by deriving a recurrence relation for the number of FQ-legal decompositions. Specifically, consider the following definitions.

\begin{itemize}

\item $d_n$: the number of FQ-legal decompositions using only elements of $\{q_1, q_2, \dots, q_n\}$. Note we include one empty decomposition of 0 in this count. Further, some of the decompositions are of numbers larger than $q_{n+1}$ (for example, for $n$ large $q_n + q_{n-2} + q_{n-20} > q_{n+1}$). We set $d_0 = 1$.

\item $c_n$: the number of FQ-legal decompositions using only elements of $\{q_1, q_2, \dots, q_{n}\}$ \emph{and} $q_{n}$ is one of the summands. We set $c_0 = 1$.

\item $b_n$: the number of FQ-legal decompositions using only elements of $\{q_1, q_2, \dots, q_{n}\}$ \emph{and} both $q_n$ and $q_{n-2}$ are used.

\end{itemize}

By brute force one can compute the first few values of these sequences; see Table \ref{table:valuesdcb}.

\begin{center}
\begin{table}[h]
\begin{center}
\begin{tabular}{|r||r|r|r||r|}
\hline
$n$ & $d_n$ & $c_n$ & $b_n$ & $q_n$\\
  \hline
\hline
  1 & 2 & 1 & 0 & 1 \\
  2 & 3 & 1 & 0 & 2 \\
  3 & 4 & 1 & 0 & 3\\
  4 & 6 & 2 & 1 & 4\\
  5 & 8 & 2 & 1 & 5\\
  6 & 11 & 3 & 1& 7 \\
  7 & 15 & 4 & 1& 9 \\
  8 & 21 & 6 & 2& 12 \\
  9 & 30 & 9 & 3& 16 \\
  10 & 42  & 12 & 4& 21 \\
  11 & 59 & 17 & 6& 28 \\
  12 & 82 & 23 & 8& 37 \\
  13 & 114 & 32 & 11& 49 \\
  \hline
\end{tabular}
\caption{Values of the first few terms of $d_n$, $c_n$ and $b_n$; for ease of comparison we have included $q_n$ as well.}\label{table:valuesdcb}
\end{center}
\end{table}
\end{center}

We first find three recurrence relations interlacing our three unknowns.

\begin{lem} For $n \ge 7$ we have
\bea\label{eq:dncnbnrelations} d_n & \ = \ & c_n + c_{n-1} + \cdots + c_0  \ = \ c_n + d_{n-1} \nonumber\\ c_n & \ = \ & d_{n-5} + c_{n-2} - b_{n-2} \nonumber\\ b_n &\ =\ & d_{n-7} , \eea which implies \be\label{eq:relationusingonlyd} d_n \ = \ d_{n-1} + d_{n-2}  - d_{n-3} + d_{n-5} - d_{n-9}. \ee \end{lem}

\begin{proof} The relation for $d_n$ in \eqref{eq:dncnbnrelations} is the simplest to see. The left hand side counts the number of FQ-legal decompositions where the largest element used is $q_n$, which may or may not be used. The right hand side counts the same quantity, partitioning based on the largest index used. It is important to note that $c_0$ is  included and equals 1, as otherwise we would not have the empty decomposition (corresponding to an FQ-legal decomposition of 0). We immediately use this relation with $n-1$  for $n$ to replace $c_{n-1} + \cdots + c_0$ with $d_{n-1}$.

Our second relation comes from counting the number of FQ-legal decompositions where $q_n$ is used and no larger index occurs, which is just $c_n$. Since $q_n$ occurs in all such numbers we cannot use $q_{n-1}, q_{n-3}$ or $q_{n-4}$, but $q_{n-2}$ may or may not be used. If we do not use $q_{n-2}$ then we are left with choosing FQ-legal decompositions where the largest index used is at most $n-5$; by definition this is $d_{n-5}$. We must add back all the numbers arising from decompositions using  $q_n$ and $q_{n-2}$. Note that if $n-2$ was the largest index used then the number of valid decompositions is $c_{n-2}$; however, this includes $b_{n-2}$ decompositions where we use both $q_{n-2}$ and $q_{n-4}$. As we \emph{must} use $q_n$, we cannot use $q_{n-4}$ and thus these $b_{n-2}$ decompositions should not have been included; thus $c_n$ equals $d_{n-5} + c_{n-2} - b_{n-2}$. (Note: alternatively one could prove the relation $c_n= d_{n-5} + b_n$.)

Finally, consider $b_n$. This counts the times we use $q_n$ (which forbids us from using $q_{n-1}, q_{n-3}$ and $q_{n-4}$) and $q_{n-2}$ (which forbids us from using $q_{n-3}, q_{n-5}$ and $q_{n-6}$). Note all other indices at most $n-7$ may or may not be used, and no other larger index can be chosen. By definition the number of valid choices is $d_{n-7}$.

We now easily derive a recurrence involving just the $d$'s. The first relation yields $c_n = d_n - d_{n-1}$ while the third gives $b_n = d_{n-7}$. We can thus rewrite the second relation involving only $d$'s, which immediately gives \eqref{eq:relationusingonlyd}. \end{proof}

Armed with the above, we solve the recurrence for $d_n$.




\begin{lem}\label{lem:growthofdn} We have \be d_n \ = \ \beta_1 r_1^n \left[1 + O\left((r_2/r_1)^n\right)\right], \ee where $\beta_1 > 0$, $r_1 \approx 1.39704$ and $r_2 \approx 1.07378$ are the two largest (in absolute value) roots of $r^7 - r^6 - r^2  -1 = 0$.
\end{lem}

\begin{proof} The characteristic polynomial associated to the recurrence for $d_n$ in \eqref{eq:relationusingonlyd} factors as \be r^9 - r^8 - r^7 + r^6 - r^4 + 1 \ = \ (r-1)(r+1)(r^7 - r^6 - r^2  -1). \ee The roots of the septic are all distinct, with the largest $r_1$ approximately 1.39704 and the next two largest being complex conjugate pairs of size $r_2 \approx 1.07378$; the remaining roots are at most 1 in absolute value. Thus by standard techniques for solving recurrence relations \cite{Gol} (as the roots are distinct) there are constants such that \be d_n \ = \ \beta_1 r_1^n + \beta_2 r_2^n + \cdots + \beta_7 r_7^n + \beta_8 1^n + \beta_9 (-1)^n. \ee

To complete the proof, we need only show that $\beta_1 > 0$ (if it vanished, then $d_n$ would grow slower than one would expect). As the roots come from a degree 7 polynomial, it is not surprising that we do not have a closed form expression for them. Fortunately a simple comparison proves that $\beta_1 >  0$. Since $d_n$ counts the number of FQ-legal decompositions using indices no more than $q_n$, we must have $d_n \ge q_n$. As $q_n$ grows like $\lambda_1^n$ with $\lambda_1 \approx 1.3247$, if $\beta_1 = 0$ then $d_n < q_n$ for large $n$, a contradiction. Thus $\beta_1 > 0$.\end{proof}


We can now determine the average behavior of $\dfq(m)$, the number of FQ-legal decompositions of $m$.

\begin{proof}[Proof of Theorem \ref{thm:growthratenumberdecomp}]  We have \be\label{eqn:ave} \dave(n) \ = \ \frac{1}{q_{n+1}} \sum_{m=0}^{q_{n+1}-1} \dfq(m). \ee

We first deal with the upper bound. The summation on the right hand side of Equation \eqref{eqn:ave} is less than $d_n$, because $d_n$ counts some FQ-legal decompositions that exceed $q_{n+1}$. Thus \be \dave(n)\ \le\ \frac{d_n}{q_{n+1}}. \ee For $n$ large by Lemma \ref{lem:growthofdn} we have \be d_n\ =\ \beta_1 r_1^n \left[1 + O\left((r_2/r_1)^n\right)\right]\ee with $\beta_1 > 0$ and $r_1 \approx 1.39704$, and from Proposition \ref{prop:fibquiltnumbers} \be q_n \ = \ \alpha_1 \lambda_1^n \left[1 + O\left((\lambda_2/\lambda_1)^n\right)\right] \ee where $\alpha_1 \approx 1.26724$, \be \lambda_1 \ = \ \frac13 \left(\frac{27}{2} - \frac{3 \sqrt{69}}{2}\right)^{1/3} + \frac{\left(\frac12 \left(9 + \sqrt{69}\right)\right)^{1/3}}{3^{2/3}} \ \approx \ 1.32472 \ee and $\lambda_2 \approx -0.662359 - 0.56228i$ (which has absolute value approximately 0.8688). Thus there is a $C_2 > 0$ such that for $n$ large we  have $\dave(n) \le C_2 (r_1/\lambda_1)^n$.


We now turn to the lower bound for $\dave(n)$. As we are primarily interested in the growth rate of $\dave(n)$ and not on optimal values for the constants $C_1$ and $C_2$, we can give a simple argument which suffices to prove the exponential growth rate, though at a cost of a poor choice of $C_1$. Note that for large $n$ the sum on the right side of Equation \eqref{eqn:ave} is clearly at least $d_{n-2016}$. To  see this, note $d_{n-2016}$ counts the number of FQ-legal decompositions using no summand larger than $q_{n-2016}$,  and if $q_{n-2016}$ is our largest summand then by \eqref{eq:sum} our number cannot exceed \be \sum_{i=1}^{n-2016} q_i \ =\  q_{n-2011}-6 \ \le \ q_n. \ee  Thus \be \dave(n) \ \ge \ \frac{d_{n-2016}}{q_{n+1}}.\ee We now argue as we did for the upper bound, noting that for large $n$ we have \be d_{n-2016} \ = \ r_1^{-2016} \cdot \beta_1 r_1^n \left[1 + O\left((r_2/r_1)^n\right)\right].\ee Thus for $n$ sufficiently large \be \dave(n) \ \ge \ C_1 (r_1/\lambda_1)^n,\ee completing the  proof.
\end{proof}

\section{Greedy Algorithms for the Fibonacci Quilt Sequence}\label{sec:greedy}

\subsection{Greedy Decomposition}

Let $h_n$ denote the number of integers from 1 to $q_{n+1}-1$ where the greedy algorithm successfully terminates in a legal decomposition.  We have already seen that the first number where the greedy algorithm fails is 6; the others less than 200 are 27, 34, 43, 55, 71, 92, 113, 120, 141, 148, 157, 178, 185 and 194.

Table \ref{table:numdecompgreedy} lists $h_n$ for the first few values of $n$, as well as $\rho_n$ the percentage of integers in $[1,q_{n+1})$ where the greedy algorithm yields a legal decomposition.

\begin{center}
\begin{table}[h]
\begin{center}
\begin{tabular}{|r||r|r||r|}
\hline
$n$ & $q_n$ & $h_n$ & $\rho_n$\\
  \hline
\hline
1	&	1	&	1	&	100.0000\\
2	&	2	&	2	&	100.0000\\
3	&	3	&	3	&	100.0000\\
4	&	4	&	4	&	100.0000\\
5	&	5	&	5	&	83.3333\\
6	&	7	&	7	&	87.5000\\
7	&	9	&	10	&	90.9091\\
8	&	12	&	14	&	93.3333\\
9	&	16	&	19	&	95.0000\\
10	&	21	&	25	&	92.5926\\
11	&	28	&	33	&	91.6667\\
12	&	37	&	44	&	91.6667\\
13	&	49	&	59	&	92.1875\\
14	&	65	&	79	&	92.9412\\
15	&	86	&	105	&	92.9204\\
16	&	114	&	139	&	92.6667\\
17	&	151	&	184	&	92.4623\\
  \hline
\end{tabular}
\caption{Values of the first few terms of $q_n$, $h_n$ and $\rho_n$.}\label{table:numdecompgreedy}
\end{center}
\end{table}
\end{center}

We start by determining a recurrence relation for $h_n$.

\begin{lem} For $h_n$ as above, \be\label{eq:recurrencehn} h_n \ = \ h_{n-1} + h_{n-5}+1, \ee with initial values $h_k = k$ for $1 \le k \le 5$. \end{lem}

\begin{proof} We can determine the number integers in $[1,q_{n+1})$ for which the greedy algorithm is successful by counting the same thing in $[1,q_{n})$ and in $[q_n,q_{n+1})$. The number of integers in $[1,q_{n})$ for which the greedy algorithm is successful is just $h_{n-1}$.

Integers  $m\in [q_n,q_{n+1})$ for which the greedy algorithm is successful must have largest summand $q_n$. So $m=q_n+x$. We claim $x\in [0,q_{n-4})$. Otherwise $m=q_n+x\geq q_n + q_{n-4}=q_{n+1}$, which is a contradiction. If $x=0$, then $m=q_n$ can be legally decomposed using the greedy algorithm and we must add 1 to our count. If $m$ is to have a successful legal greedy decomposition then so must $x$. Hence it remains to count how many $x\in [1,q_{n-4})$ have successful legal greedy decompositions, but this is just $h_{n-5}$. Combining these counts finishes the proof.
\end{proof}

We now prove the greedy algorithm successfully terminates for a positive percentage of integers, as well as fails for a positive percentage of integers.

\begin{proof}[Proof of Theorem \ref{thm:successgreedyalg}] Instead of solving the recurrence in \eqref{eq:recurrencehn}, it is easier to let $g_n = h_n + 1$ and first solve \be g_n \ = \ g_{n-1} + g_{n-5}, \ \ \ g_k \ = \ k+1 \ {\rm for}\ 1 \le k \le 5. \ee The characteristic polynomial for this is \be r^5 - r^4 - 1 \ = \ 0, \ \ \ {\rm or} \ \ \ (r^3 - r - 1) (r^2 - r + 1). \ee
By standard recurrence relation techniques, we have \be\label{eqn:gn} g_n \ = \ c_1 \lambda_1^n + c_2 \lambda_2^n + \cdots + c_5 \lambda_5^n, \ee where \be \lambda_1 \ = \ \frac13 \left(\frac{27}{2} - \frac{3 \sqrt{69}}{2}\right)^{1/3} + \frac{\left(\frac12 \left(9 + \sqrt{69}\right)\right)^{1/3}}{3^{2/3}} \ \approx \ 1.32472 \ee is the largest root of the recurrence for $g_n$ (the other roots are at most 1 in absolute value).

By Proposition \ref{prop:fibquiltnumbers} we have
\be q_n \ = \ \alpha_1 \lambda_1^n + \alpha_2 \lambda_2^n + \alpha_3 \lambda_3^n, \ee
where $\lambda_1,\lambda_2,\lambda_3$ are the same as in Equation \eqref{eqn:gn} and $\alpha_1\approx1.26724$.

 We must show that $c_1 \alpha_1 \neq 0$, as this will imply that $g_n$ and $q_n$ both grow at the same exponential rate.
As $g_n \ge 2 g_{n-5}$ implies $g_n \ge c 2^{n/5}$ we have that $g_n$ is growing exponentially,  thus $c_1 \neq 0$.

Unfortunately writing  $c_1$ in closed form requires solving a fifth order equation, but this can easily be done numerically and the limiting ratio $\rho_n = h_n/(q_{n+1}-1)$ can be approximated well. That ratio converges to $\frac{c_1}{\alpha_1} \frac1{\lambda_1} \approx 0.92627$. \end{proof}


\subsection{Greedy-6 Decomposition}

\ \\


\begin{lemma}\label{lem:5sum}
For $\ell\geq 1+5k$ and $k\geq0$, we have $q_\ell+q_{\ell-5}+\cdots+q_{\ell-5k}<q_{\ell+1}$.
\end{lemma}

\begin{proof}
We proceed by induction on $k$. For the Basis Step, note \be q_\ell+q_{\ell-5}\ < \ q_\ell+q_{\ell-4} \ = \ q_{\ell+1}.\ee
For the Inductive Step: By inductive hypothesis and the recurrence relation stated in Theorem \ref{thm:rrfibquilt},
\be
q_\ell+(q_{\ell-5}+\cdots+q_{\ell-5k})<q_\ell+q_{\ell-4} \ = \  q_{\ell+1},
\ee completing the proof.
\end{proof}

\begin{proof}[Proof of Theorem \ref{thm:greedy6}]
For the first part, we verify that if $m \leq 151 = q_{17}$ the theorem holds. Define $I_n:=[q_n,q_{n+1}) = [q_n,q_{n+1}-1].$  Assume for all $m \in \displaystyle\cup_{\ell=1}^{n-1}I_\ell$, $m$ satisfies the theorem.  Now consider $m \in I_{n}$. If $m=q_n$ then we add done.  Assume  $m = q_{n} + x$ with $x > 0$. Since $q_{n+1} = q_n + q_{n-4},$ we know $x < q_{n-4}.$  Then by the inductive hypothesis we know the $x$ satisfies the theorem.  Namely, $\mathcal{G}(x) = q_{k_1} + q_{k_2} + \dots + q_{k_s}$ is a FQ-legal decomposition which satisfies either Condition (1) or (2) but not both.  Then $\mathcal{G}(m) = q_n + q_{k_1} + q_{k_2} + \dots + q_{k_s}$ and lastly $n-k_1 \geq 5$.\\ \

For the second part, let $m=q_{\ell_1}+q_{\ell_2}+\cdots+q_{\ell_{t-1}}+q_{\ell_t}$ be a decomposition that satisfies either Condition (1) or (2) but not both. Note that in both cases, this decomposition is legal. If $t=1$, then $m$ is a Fibonacci Quilt number and the theorem is trivial. So we assume $t\geq 2$. Hence by construction of the sequence, $m$ is not a Fibonacci Quilt number.

Let $\mathcal{G}(m) = q_{k_1}+q_{k_2}+\cdots+q_{k_s}$. Note that $s\geq2$. For contradiction we assume the given decomposition   is not the Greedy-6 decomposition. Without loss of generality we may assume $q_{\ell_1}\neq q_{k_1}$. Since $q_{k_1}$ was chosen according to the Greedy-6 algorithm,  $q_{\ell_1}< q_{k_1}$.\\ \

\noindent {\bf Case 1}: Using Lemma \ref{lem:5sum},
\begin{align}
m& \ = \  q_{\ell_1}+q_{\ell_2}+\cdots+q_{\ell_{t-1}}+q_{\ell_t}
\leq q_{\ell_1}+q_{\ell_1-5}+\cdots+q_{\ell_{1}-5(t-1)}
<q_{\ell_1+1}
\leq q_{k_1}
< m
\end{align}
which is a contradiction.\\ \

\noindent {\bf Case 2}: Again using Lemma \ref{lem:5sum},
\begin{align}
m&\ =\ q_{\ell_1}+q_{\ell_2}+\cdots+q_{\ell_{t-2}}+q_{4}+q_2
\ =\ q_{\ell_1}+q_{\ell_2}+\cdots+q_{\ell_{t-2}}+q_{5}+q_1\nonumber\\
&\ \leq\ q_{\ell_1}+q_{\ell_1-5}+\cdots+q_{\ell_{1}-5(t-2)}+q_1\nonumber\\
&\ <\ q_{\ell_1+1}+q_1\nonumber\\
&\ \leq\ q_{k_1}+q_1\nonumber\\
&\ \leq\ m
\end{align}
which is a contradiction.
\end{proof}

In order to prove Theorem \ref{thm:Dm} we will need several relationships between the terms in the Fibonacci Quilt sequence. The following lemma describes those relationships.

\begin{lemma}\label{lem:FQreln} The following hold.

\begin{enumerate}
\item If $n\geq 7$, then $2q_n  \ = \ q_{n+2}+q_{n-5}.$
\item If $n\geq 8$, then $q_n + q_{n-2 }  \ =\ q_{n+1} + q_{n-5}.$
\item If $n\geq 10$, then $q_n + q_{n-3 }  \ =\ q_{n+1} + q_{n-8}.$
\end{enumerate}
\end{lemma}

\begin{proof} The proof follows from repeated uses of the recurrence relations stated in Theorem \ref{thm:rrfibquilt}:
\bea
2q_n\ = \ q_n+q_{n-1}+q_{n-5}  \ = \ q_{n+2}+q_{n-5},
\eea
\bea
q_{n} + q_{n-2}   \ = \  q_n + q_{n-4} + q_{n-5}\ = \ q_{n+1} + q_{n-5},
\eea
and
\bea
q_n + q_{n-3} \ = \ q_n + q_{n-4} + q_{n-3} - q_{n-4}\ = \ q_{n+1} + q_{n-8}.
\eea
\end{proof}

\begin{proof}[Proof of Theorem \ref{thm:Dm}] The proof follows by showing that we can move from $\mathcal{D}(m)$ to $\mathcal{G}(m)$ without increasing the number of summands by doing five types of moves. That the summation remains unchanged after each move follows from Lemma \ref{lem:FQreln} and Theorem \ref{thm:rrfibquilt}.  \\ \

\begin{enumerate}

\item Replace $2q_{n} $ with $q_{n+2}+q_{n-5}$ (for $n \ge 7$). (If $n\le 6$, replace $2q_6$ with $q_8+q_2$ , replace $2q_5$ with $q_7+q_1$, replace $2q_4$ with $q_6+q_1$, replace $2q_3$ with $q_5+q_1$, replace $2q_2$ with $q_4$, and replace $2q_1$ with $q_2$.)\\ \

\item Replace $q_{n-1} + q_{n-2}$ with $q_{n+1}$ (for $n \ge 5$). In other words, if we have two adjacent terms, use the recurrence relation to replace. (If $n\le 4$, replace $q_3+q_2$ with $q_5$ and replace $q_2+q_1$ with $q_3$.) \\ \

\item Replace $q_n+q_{n-2}$ with $q_{n+1}+q_{n-5}$  (for $n \ge 8$).     (If $n\le 7$, replace  $q_7+q_5$ with $q_8+q_2$, $q_6+q_4$ with $q_7+q_2$, $q_5+q_3$ with $q_7+q_1$, $q_4+q_2$ with $q_5+q_1$, and $q_3+q_1$ with $q_4$.)  \\ \

\item Replace $q_n+q_{n-3}$ with $q_{n+1}+q_{n-8}$  (for $n \ge 10$).   (If $n\le 9$, replace $q_9+q_6$ with $q_{10}+q_2$, $q_8+q_5$ with $q_9+q_1$, $q_7+q_4$ with $q_8+q_1$, $q_6+q_3$ with $q_7+q_1$, $q_5+q_2$ with $q_6$, and $q_4+q_1$ with $q_5$.)  \\ \

\item Replace $q_{n} + q_{n-4}$ with $q_{n+1}$ (for $n \ge 6$). In other words, if we have two adjacent terms, use the recurrence relation to replace.\\ \ 

\end{enumerate}

Notice that in all moves, the number of summands either decreases by one or remains unchanged. In addition, the sum of the indices either decreases or remains unchanged.  There are three situations where neither the index sum nor the number of summands decreases; $q_5+q_3=q_7+q_1$, $q_4+q_2=q_5+q_1$, and $2q_3 = q_5+q_1$. But in these situations,  the number of $q_5, \,q_4,\, q_3, \, q_2$ decrease.
Therefore this process eventually terminates because the index sum and the number of summands cannot decrease indefinitely.

Let $m=q_{\ell_1}+q_{\ell_2}+\cdots+q_{\ell_{t-1}}+q_{\ell_t}$ be the decomposition obtained after all possible moves. Each move  either decreases the number of summands or replaces two summands with two  that are farther apart in the sequence. In fact, closer examination of the moves reveals
$\ell_i-\ell_{i-1}\geq 5$ except maybe $\ell_{t-1}=5$ and $\ell_t = 1$.

If $\ell_{t-1}=5$ and $\ell_t = 1$, replace $q_5+q_1$ with $q_4+q_2$. By Theorem \ref{thm:greedy6} this is the Greedy-6 decomposition of $m$.
\end{proof}

\appendix

\section{Generalized Binet Formula for $(s,b)$-Generacci Sequence}\label{sec:genbinetformSBGen}


We now prove the Generalized Binet Formula for the $(s,b)$-Generacci sequence. The argument is \emph{almost} standard, but the fact that the leading coefficient in the recurrence relation is zero leads to some technical obstructions. We resolve these by first passing to a related characteristic polynomial where the leading coefficient is positive (and then Perron-Frobenius arguments are applicable), and then carefully expand to our sequence.

\begin{proof}[Proof of the Generalized Binet Formula in Theorem \ref{thrm:recurrence}] The recurrence in \eqref{recurrence} generates the  characteristic polynomial
\begin{align}
x^{(s+1)b}-x^{sb}-b\ = \ 0.
\label{sb-poly}
\end{align}
Letting $y=x^b$ in \eqref{sb-poly}, we are able to pass to studying
\begin{equation}
q(y) \ = \  y^{s+1} - y^s - b \ = \  0.
\label{auxiliary-poly}
\end{equation}

The polynomial $q(y)$ has the following properties.
\begin{enumerate}
\item[(1)] The roots are distinct.
\item[(2)] There is a positive root $r$ satisfying $r  > |r_j|$ where $r_j$ is any other  root of $q(y)$.
\item[(3)] The positive root $r$ described in (2) satisfies $r > 1$ and is the only positive root.
\end{enumerate}

To prove  property (1), consider
\begin{equation}
q'(y) \ = \  (s+1)y^s - sy^{s-1} \ = \  (s+1)y^{s-1}\left(y - \displaystyle \frac{s}{s+1}\right).
\end{equation}
If a repeated root $y$ exists then $q(y)=q'(y)= 0$. Clearly, $y=0$ is not a root, so $y=\displaystyle \frac{s}{s+1} < 1$. In this case,
\begin{equation}
b \ = \  \left(\displaystyle \frac{s}{s+1}\right)^s\left(\displaystyle \frac{s}{s+1} -1\right)\ <\ 0,
\end{equation}
which is a contradiction as $b$ is a positive integer.

Property (2) follows from the same argument used in the proof of Theorem A.1 in \cite{BBGILMT}, or by using the Perron-Frobenius Theorem for non-negative irreducible matrices.

Furthermore, since the root $r$ satisfies $r^s(r-1) = b$ and $b >0$, necessarily $r >1$. Now, $q(0) = -b, \ q(r) = 0,   \ q'(y) < 0$ for $y< \displaystyle \frac{s}{s+1}$, and $q'(y) > 0$ for $y > \displaystyle \frac{s}{s+1}$, implies that $q(y) > 0$ for all $y > r$. Hence, $r$ is the only positive root of $q(y)$, completing the proof of property (3).

Let $\omega_1  > 0$ be chosen so that $\omega_1^b =r$, and let the (distinct) roots of \eqref{auxiliary-poly} be denoted by $\omega_1^b, \omega_2^b, \ldots, \omega_{s+1}^b$, where $\omega_1^b > 1$ is the only positive root and $\omega_1^b > |\omega_j^b|$, for all $j=2, \ldots, s+1$.  For convenience, we arrange the roots  so that  $\omega_1^b > |\omega_2^b| \ge  \cdots \ge |\omega_{s+1}^b|$. Then the roots of \eqref{sb-poly} are given by
\begin{equation}
\omega_1,\, \omega_1 \zeta_b,\, \ldots,\, \omega_1 \zeta_b^{b-1},\,\omega_2,\, \omega_2 \zeta_b, \,\ldots, \,\omega_2 \zeta_b^{b-1}, \,\ldots,\, \omega_{s+1}, \,\omega_{s+1} \zeta_b, \,\ldots, \,\omega_{s+1} \zeta_b^{b-1},
\end{equation}
where $\zeta_b = e^{\frac{2\pi i}{b}}$ is a primitive $b$\textsuperscript{th} root of unity. Now, using standard results on solving linear recurrence relations (see for example \cite[Section 3.7]{Gol}), the $n$\textsuperscript{th} term of the sequence has an expansion
\begin{align}
a_n    \ = \   \sum_{k=0}^{b-1} \sum_{j=1}^{s+1} \alpha_{k,j}(\omega_j \zeta_b^k)^n \ = \     \sum_{j=1}^{s+1} \left(\sum_{k=0}^{b-1}\alpha_{k,j} \zeta_b^{kn}\right)\omega_j^n,
\end{align}
for some constants $\alpha_{k,j}$.

For $n=\ell b  + v, \ v=0,1,\ldots,b-1$,
\begin{equation}\zeta_b^{k(\ell b  + v)} \ = \  (\zeta_b^{b})^{\ell k} (\zeta_b^{vk}) \ = \  \zeta_b^{v k}, \  \  \  \mbox{for any } k.\end{equation}
Thus
\begin{align}
a_{\ell b  + v}    \ = \    \sum_{j=1}^{s+1} \left(\sum_{k=0}^{b-1}\alpha_{k,j}\zeta_b^{vk} \right) (\omega_j^{b})^{\ell}\omega_j^{v} \ = \     \sum_{j=1}^{s+1} c_j (\omega_j^{b})^{\ell},
\end{align}
where $c_j = \displaystyle \omega_j^{v} \sum_{k=0}^{b-1} \alpha_{k,j}\zeta_b^{vk}$ \  \ for $j=1, \ldots, s+1$.

Note that  $c_1$ must be  a real number, as otherwise $a_{\ell b  + v}$ is non-real for large $\ell$  (since $\omega_1^{b{\ell}} > 0$ is the dominant term in the expansion). The final step is to prove that $c_1 > 0$. If $c_1 < 0$, then  for large $\ell$,  $a_{\ell b  + v}< 0$  (again since $\omega_1^{b{\ell}} > 0$ is the dominant term in the expansion).  If $c_1 = 0$, then $a_{\ell b  + v}= \displaystyle \sum_{j=m}^{s+1} c_j \omega_j^{b\ell}$, where $m$ is the smallest index greater than 1 such that $c_{m} \neq 0$. Then the dominant term in the expansion is $\omega_m^{b  \ell}$,  where,  by property (3) of the polynomial $q(y)$,  the root $\omega_m^b$ is either negative or complex non-real. If $\omega_m^b < 0$, then $\omega_m^{b  \ell}$ alternates in sign which violates $a_{\ell b  + v}> 0$ for all $\ell$. If $\omega_m^b$ is complex nonreal, then $\omega_m^{b  \ell}$ is not always real, again  violating $a_{\ell b  + v}> 0$ for all $\ell$.  Thus, $c_j = 0$ for all $j > 1$, and since $c_1=0$ this implies that  $a_{\ell b  + v}= 0$, a contradiction.  \end{proof}


\ \\

\end{document}